\documentclass[11pt,twoside]{article}
\usepackage{amsmath, amsthm, amssymb}
\newtheorem{definition}{Definition}
\newtheorem{theorem}{Theorem}
\newtheorem{lemma}{Lemma}
\newtheorem{remark}{Remark}
\newtheorem{example}{Example}
\newtheorem{corollary}{Corollary}
\makeatletter \oddsidemargin0.45in \evensidemargin \oddsidemargin
\begin{document}
\thispagestyle{empty} \setcounter{page}{1}



\begin{center}
{\Large\bf  Banach Contraction Principle for Cyclical Mappings on
Partial Metric Spaces}

\vskip.20in

T. Abdeljawad$^{a,b,}$\footnote{Corresponding Author E-Mail
Address: thabet@cankaya.edu.tr}, J. O. Alzabut$^{b}$, A.
Mukheimer$^b$, Y. Zaidan$^{b,c}$ \\[2mm]
{\footnotesize $^a$Department of
Mathematics, \c{C}ankaya University, 06530, Ankara, Turkey}\\
{\footnotesize $^b$Department of Mathematics and Physical
Sciences,  Prince Sultan University\\
{\footnotesize P. O. Box 66833, Riyadh 11586, Saudi Arabia}\\
{\footnotesize $^c$Department of Mathematics, University of Wisconsin--Fox Valley}\\
{\footnotesize Menasha, WI 54952, USA}}
\end{center}

\vskip.2in

{\footnotesize \noindent {\bf Abstract.}  In this paper, we prove
that the Banach contraction principle proved by S. G. Matthews in
1994 on $0$--complete partial metric spaces can be extended to
cyclical mappings. However, the generalized contraction principle
proved by D. Ili\'{c}, V. Pavlovi\'{c} and V. Rako\u{c}evi\'{c} in
"Some new extensions of Banach's contraction principle to partial
metric spaces, Appl. Math. Lett. 24 (2011), 1326--1330" on
complete partial metric spaces can not be extended to cyclical
mappings. Some examples are given to illustrate the effectiveness
of our results. Moreover, we generalize some of the results
obtained by W. A. Kirk, P. S. Srinivasan and P. Veeramani in
"Fixed points for mappings satisfying cyclical contractive
conditions, Fixed Point Theory  4 (1) (2003),79--89". Finally, an
Edelstein's type theorem is also extended in case one of the sets
in the cyclic
decomposition is $0$-compact.  }

\vskip.1in

\section{Introduction and Preliminaries} \label{s:1}
The partial metric spaces were first introduced in \cite{Mat94} as
a part of the study of non--symmetric topology, domain theory and
denotational semantics of dataflow networks. In particular, the
author established the precise relationship between partial metric
spaces and the so--called weightable quasimetric spaces and proved
a partial metric generalization of Banach contraction mapping
theorem which is considered to be the core of many extended fixed
point theorems; we refer the reader to the papers
\cite{Mat94,TEKcommon,TEKcontrol,Thweakly,Mat92,0V2004,02005,IATOP2010,IAFTPA2011,wasfi}.

The widespread applications of the notion of partial metric spaces
in programming theory have attracted the attention of many authors
who recently published important results in the direction of
generalizing this principle; see for instance
\cite{Khan,Rhoades,Dutta,BW1969}.  The contraction type conditions
used in these generalizations, however, do not apparently reflect
the structure of partial metric spaces. In the remarkable paper
\cite{Rako}, the authors proved more appropriate contraction
principle in partial metric spaces. Indeed, it is more convenient
to call the contraction type condition used in this paper by
partial contractive condition.

In this paper, we prove that the Banach contraction principle
obtained in \cite{Mat94} on $0$--complete partial metric spaces
can be extended to cyclical mappings. However, the generalized
contraction principle proved in \cite{Rako} on complete partial
metric spaces can not be extended to cyclical mappings. Some
examples are given to illustrate the effectiveness of our results.
In addition to this, we generalize some of the results obtained in
\cite{Kirk}. Finally, an Edelstein's type theorem is also extended
in case one of the sets in the cyclic decomposition is
$0$-compact.

We recall some definitions of partial metric spaces and state some
of their properties. A partial metric space (PMS) is a pair $(X,
p:X\times X\rightarrow\mathbb R^+)$ (where $\mathbb R^+$ denotes
the set of all non negative real numbers) such that
\begin{itemize}
\item[]
\begin{itemize}
    \item [(P1)] $p(x,y) =p(y,x)$\;\;(symmetry);
    \item [(P2)] If $0\leq p(x,x) = p(x,y) = p(y,y)$ then $x=y$\;\;
(equality);
    \item [(P3)] $p(x,x) \leq p(x,y)$\;\;   (small
    self--distances);

    \item [(P4)] $p(x,z) + p(y,y) \leq p(x,y) + p(y,z)$\;\;
(triangularity);
\end{itemize}
\end{itemize}
for all $x,y,z \in X$.

For a partial metric $p$ on $X$,  the function $p^s : X \times X
\rightarrow \mathbb R^+$ given by
\begin{equation}
p^s(x,y)=2p(x,y)-p(x,x)-p(y,y) \label{elma1}
\end{equation}
is a (usual) metric on $X$. Each partial metric $p$ on $X$
generates a $T_0$ topology $\tau_p$ on $X$ with a base of the
family of open $p$-balls $\{B_p(x, \varepsilon) : x\in X,
\varepsilon> 0\}$, where $B_p(x, \varepsilon) =\{y \in X : p(x, y)
< p(x, x) +\varepsilon\}$ for all $x \in X$ and $\varepsilon >0 $.

\begin{definition}
\emph{\cite{Mat94}}
\begin{itemize}
\item[$(i)$] A sequence $\{x_n\}$ in a PMS $(X,p)$ converges to
$x\in X$ if and only if $p(x,x)= \lim_{n \rightarrow \infty}
p(x,x_n)$. \item[$(ii)$] A sequence $\{x_n\}$ in a PMS $(X,p)$ is
called a Cauchy if and only if  $\lim_{n,m\rightarrow \infty}
p(x_n, x_m)$ exists (and finite). \item[$(iii)$] A PMS $(X,p)$ is
said to be complete if every Cauchy sequence $\{x_n\}$ in $X$
converges, with respect to $\tau_p$, to a point $x\in X$ such that
$p(x, x)=\lim_{n,m_\rightarrow \infty} p(x_n, x_m)$. \item[$(iv)$]
A mapping $f:X \rightarrow X$ is said to be continuous at $x_0 \in
X$, if for every $ \varepsilon> 0$, there exists $\delta > 0$ such
that $f(B_p(x_0,\delta))\subset B_p(f(x_0),\varepsilon)$.
\end{itemize}
\end{definition}

\begin{lemma}\emph{\cite{Mat94}}
\begin{enumerate}
\item[$(a1)$] A sequence $\{x_n\}$ is Cauchy in a PMS $(X,p)$ if
and only if $\{x_n\}$ is Cauchy in a metric space $(X,p^s)$.
\item[$(a2)$] A  PMS $(X,p)$ is complete  if and only if the
metric space $(X,p^s)$ is complete. Moreover,
\begin{equation}
\lim_{n \rightarrow \infty}p^s(x,x_n)=0 \Leftrightarrow p(x,
x)=\lim_{n \rightarrow \infty}p(x,x_n)=\lim_{n,m_\rightarrow
\infty} p(x_n, x_m).
 \label{Mat}
\end{equation}

\end{enumerate}
\end{lemma}

\begin{lemma} \label{cont}
Let $(X,p)$ be a partial metric space and let $T:X\rightarrow X$
be a continuous self-mapping. Assume $\{x_n\}\in X$ such that
$x_n\rightarrow z~\texttt{as}~n\rightarrow \infty$. Then
$$\lim_{n\rightarrow \infty} p(Tx_n,Tz)=p(Tz,Tz).$$
\end{lemma}
\begin{proof}
Let $\epsilon > 0$ be given. Since $T$ is continuous at $z$ find
$\delta > 0$ such that $T(B_p(z,\delta))\subseteq
B_p(Tz,\epsilon)$. Since $x_n\rightarrow z$ then
$\lim_{n\rightarrow \infty}p(x_n,z)=p(z,z)$ and hence find $n_0
\in N$ such that $p(z,z)\leq p(x_n,z)< p(z,z)+\delta$ for all
$n\geq n_0$. That is, $x_n \in B_p(z,\delta)$ for all $n\geq n_0$.
Thus $T(x_n)\in B_p(Tz,\epsilon) $ and so $p(Tz,Tz))\leq
p(Tx_n,Tz)< p(Tz,Tz)+\epsilon$ for all $n\geq n_0$. This shows our
claim.
\end{proof}

A sequence $\{x_n\}$ is called $0$--Cauchy  if $\mbox{lim}_{m,n\to
\infty} p(x_n,x_m)=0$. The partial metric space $(X,p)$ is called
$0$--complete if every $0$--Cauchy sequence in $x$ converges to a
point $x\in X$ with respect to $p$ and $p(x,x)=0$. Clearly, every
complete partial metric space is complete. The converse need not
be true; see \cite{Roma} for more details.

\begin{example} \label{neednot}
Let $X=\mathbb{Q}\cap [0,\infty)$ with the partial metric
$p(x,y)=max\{x,y\}$ where $\mathbb{Q}$ is the set of rationals.
Then $(X,p)$ is a $0$--complete partial metric space which is not
complete.
\end{example}

\begin{theorem} \emph{\cite{Mat94,Roma}}\label{Mathew}
Let$(X,p)$ be a 0-complete partial metric space and
$f:X\rightarrow X$ be such that
$$p(f(x),f(y))\leq \alpha p(x,y) \forall x,y\in X \mbox{and}~ \alpha\in [0,1)$$
there exists a unique $u\in X$ such that $u=f(u)$ and $p(u,u)=0$.
\end{theorem}

Let $\rho_p=\mbox {inf}\{p(x,y) : x,y \in X\}$ and define $X_p=\{x \in X : p(x,x)=\rho_p\}$.\\
\begin{theorem} \emph{\cite{Rako}} \label{Rako}
Let $(X,p)$ be a complete metric space, $\alpha \in [0,1)$ and $T
: X\rightarrow X$ a given mapping. Suppose that for each $x, y \in
X$ the following condition holds $$p(x,y)\leq \mbox{max}\{\alpha
p(x,y), p(x,x), p(y,y)\}.$$ Then
\begin{itemize}
\item[(1)] the set $X_p$ is nonempty; \item[(2)] there is a unique
$u \in X_p$ such that $Tu=u$; \item[(3)] for each $x \in X_p$ the
sequence $\{T^nx\}_{n\geq1}$ converges with respect to the metric
$p^s$ to $u$.
\end{itemize}
\end{theorem}
\begin{definition} \label{Cyc contraction}
Let $A$ and $B$ be two nonempty closed subsets of a complete
partial metric space $(X,p)$ such that $X=A\cup B$. A mapping
$T:X\rightarrow X$ is called cyclical contraction if it satisfies
\begin{itemize}
  \item[(C1)] $ T(A)\subseteq  B ~~and ~~T(B)\subseteq A;$
  \item[(C2)] $\exists ~0<\alpha < 1: p(Tx,Ty)\leq \alpha p(x,y),~~\forall~x\in A~and~ \forall~y\in B.$
 \end{itemize}
If $(C2)$ in Definition \ref{Cyc contraction} is replaced by the
condition
$$(PC2)~\exists ~0<\alpha < 1: p(Tx,Ty)\leq max\{\alpha p(x,y),p(x,x),p(y,y)\}~~\forall~x\in A~and~ \forall~y\in B.$$
Then $T$ is called a partial cyclical contraction.
\end{definition}
\begin{remark} The partial cyclical contractions reflects the
real structure of partial metric space.
\end{remark}
The proof of the following lemma can be easily achieved by using
the partial metric topology.

\begin{lemma} \label{sclosed}
A subset $A$ of a partial metric space is closed if and only if $x
\in A$ whenever $x_n \in A$ satisfies $x_n\rightarrow x$ as $n \to
\infty$.
\end{lemma}
\begin{definition} \label{zero compact}
A set $A$ in a partial metric space $(X,p)$ is called $0$--compact
if for any sequence $\{x_n\}$ in $A$ there exists a subsequence
$\{x_{k_n}\}$ and $x \in A$ such that $\lim_{n\rightarrow
\infty}p(x_{k_n},x)=p(x,x)=0$.
\end{definition}
Clearly a closed subset of a $0$--compact set is $0$-compact.
\begin{lemma} \emph{\cite{TEKcommon,Thweakly}} \label{continuity of pm}
Assume $x_n\rightarrow z$ as $n\rightarrow\infty$ in a PMS $(X,p)$
such that $p(z,z)=0$. Then $\lim_{n\rightarrow\infty}
p(x_n,y)=p(z,y)$ for every $y \in X$.
\end{lemma}

\section{The Main Results } \label{s:1}
We start this section by a theorem that will motivate to obtain
our main result for cyclic contraction mappings.
\begin{theorem} \label{pre}
Let $(X,p)$ be a 0-complete partial metric space and
$T:X\rightarrow X$  be continuous such that

\begin{equation}\label{pre1}
    p(Tx,T^2x)\leq \alpha p(x,Tx) ~~\forall x\in X,~~where~~ \alpha \in (0,1).
\end{equation}
Then there exists $z\in X$ such that $p(z,z)=0$ and
$p(Tz,z)=p(Tz,Tz)$.
\begin{proof}
Condition (\ref{pre1}) implies that the sequence ${T^n(x)}$ is
0-Cauchy for all $x \in X$. Hence, there exists $z \in X$ such
that $x_n=Tx_{n-1}$ converges to $z$ and $p(z,z)=0$.  The
conclusion that $p(Tz,z)=p(Tz,Tz)$ follows by Lemma \ref{cont},
(P2) and the inequality
$$p(Tz,z)\leq p(Tz,x_{n+1})+p(x_{n+1},z).$$
\end{proof}
\end{theorem}
We observe that if the partial metric in Theorem \ref{pre} is
replaced by a metric then we conclude that $z$ is a fixed point.
The following theorem is an extension of Theorem 1.1 in
\cite{Kirk}.
\begin{theorem} \label{main}
Let $A$ and $B$ be two nonempty closed subsets of a 0-complete
partial metric space $(X,p)$ such that $X=A\cup B$, and suppose
$T:X\rightarrow X$ be a cyclical contraction self--mapping of
$X$.Then $T$ has a unique fixed point in $A\cap B$.
\end{theorem}

\begin{proof}
Condition (C1) implies that for any $x \in A\cup B$
$$p(Tx,T^2x)\leq\alpha p(x,Tx)$$
and this by (P4) implies that the sequence $\{T^n(x)\}$ is
0-Cauchy for any $x \in X$. Consequently, $\{T^n(x)\}$ converges
to some point $z \in X$ such that $p(z,z)=0$. However, in view of
(C2) an infinite number of terms of the sequence $\{T^n(x)\}$ lie
in $A$ and an infinite number of terms lie in $B$. Then by Lemma
\ref{sclosed} we conclude that $z \in A\cap B$, so $A\cap B \neq
\emptyset$. Now (C1) and (C2) imply that the map $T$ restricted to
$A\cap B$ is contraction. Then the result follows by Theorem
\ref{Mathew}.
\end{proof}

In what follows, we give an example showing that the
generalization to partial metric space in Theorem \ref{main} is
proper.

\begin{example} \label{proper}
Let $X=[0,1]$, $A=[0,\frac{1}{2}]$ and $B=[\frac{1}{2},1]$. Then
$X=A\cup B$ and $A \cap B=\{\frac{1}{2}\}$. Provide $X$ with the
partial metric
$$p(x,y)=|x-y|\;\; \mbox{if \;both}\;\; x, y \in [0,1)\; \;\mbox{and}\;\;
p(x,y)=max\{x,y\}\;\; \mbox{otherwise}.$$ Then clearly $(X,p)$ is
a complete partial metric space. Define $T:X\rightarrow X$ by
$T(x)=\frac{1}{2}$ if  $0\leq x <1$ and $T(1)=0$. Then it can be
easily checked that $T$ is a cyclical contraction with $\alpha
=\frac{3}{4}$. Notice that the cyclical contractive condition of
Theorem \ref{main} is not satisfied when the partial metric $p$ is
replaced by the usual absolute value metric.

\end{example}
The following example shows that Theorem \ref{Rako} can not be
extended for cyclical  mappings when the cyclical contraction is
replaced by a partial cyclical contraction.

\begin{example} \label{ex cyc 1}
Let $A=[0,1]$,  $B=[3,4]\cup \{\frac{3}{2}\}$ and $X=A\cup B$.
Define $p:X\times X\rightarrow [0,\infty)$ by $p(x,y)=max\{x,y\}$.
Then $(X,p)$ is a complete partial metric space. Define
$T:X\rightarrow X$ by
\begin{displaymath}
T(x)=\{\begin{array}{ll}\frac{3}{2},& \textrm{$0 \le x <1$}\\
\;\;\frac{1}{2}, & \textrm{$x=\frac{3}{2}$}\\
\frac{x-2}{2}, & \textrm{$3 \le x \le 4$}
\end{array}.
\end{displaymath}
It can be easily seen that
$$p(Tx,Ty)=max\{\frac{3}{2}, \frac{y-2}{2}\}=\frac{3}{2}\leq max\{\alpha p(x,y),p(x,x),p(y,y)\}=y,$$
for any $x \in A$ , $y \in B$ and any $\alpha \in (0,1)$. However,
$ A\cap B=\emptyset$.
\end{example}

\begin{corollary} \label{common}
Let $A$ and $B$ be two nonempty closed subsets of a complete
partial metric space $(X,p)$ such that $X=A\cup B$. Let $f:A
\rightarrow B$ and $g:B\rightarrow A$ be two functions such that
$f(x)=g(x)$ for all $x \in A\cap B$ and

\begin{equation}\label{common1}
 p(f(x),g(y))\leq \alpha p(x,y)~~~ \forall x\in A~~and~~ y \in B,
\end{equation}
where $0< \alpha <1$. Then there exists a unique $x_0 \in A\cap B$
such that
$$f(x_0)=g(x_0)=x_0.$$
\end{corollary}

\begin{proof}
Apply Theorem \ref{main} to the mapping $T:A\cup B\rightarrow
A\cup B$ defined by setting
\begin{displaymath}
T(x)=\{\begin{array}{ll}f(x),& \textrm{$x \in A$}\\
\;\;g(x), & \textrm{$x \in B$}
\end{array}.
\end{displaymath}
Observe that the assumption  $f(x)=g(x)$ for all $x \in A\cap B$
implies that $T$ is well defined.
\end{proof}

\begin{remark} In the metric space case, condition (\ref{common1}) implies
that the map $T$ is well defined.
\end{remark}

Obviously Theorem \ref{main} can be extended to the following
version.

\begin{theorem}
Let $\{A_i\}_{i=1}^k$ be nonempty closed subsets of a
$0$--complete partial metric space, and suppose that
$T:\bigcup_{i=1}^k A_i\rightarrow \bigcup_{i=1}^k A_i$ satisfies
the following conditions (where $A_{k+1}=A_1$)
\begin{itemize}
\item[(1)] $T(A_i)\subseteq A_{i+1}$ for $1\leq i\leq k$;
\item[(2)] $\exists \; \alpha \in (0,1)$ such that
$p(T(x),T(y))\leq \alpha p(x,y)\; \forall x\in A_i, y\in A_{i+1}$
for $1\leq i\leq k$.
\end{itemize}
Then $T$ has a unique fixed point.
\end{theorem}
\begin{proof} One only need to observe that given $x\in
\bigcup_{i=1}^k A_i$, infinitely many terms of the Cauchy sequence
$\{T^n(x)\}$ lie in each $A_i$. Thus $\bigcap_{i=1}^k A_i\neq
\emptyset$, and the restriction of $T$ to this intersection is a
contraction mapping.
\end{proof}

\begin{remark} \label{final}
It is of our belief that Theorem \ref{main} can be extended to
more general cyclical contraction mappings. However, it would be
of more interest if the contractive type conditions are considered
with control functions.
\end{remark}
The following theorem is an extension of an Edelstein's type to
partial metric spaces.

\begin{theorem}\label{Ed}
Let $\{A_i\}_{i=1}^k$ be nonempty closed subsets of a partial
metric space $(X,p)$, at least one of which is $0$--compact, and
suppose that $T:\bigcup_{i=1}^k A_i\rightarrow \bigcup_{i=1}^k
A_i$ satisfies the following conditions (where $A_{k+1}=A_1$).
\begin{itemize}
\item[(1)] $T(A_i)\subseteq A_{i+1}$ for $1\leq i\leq k$;
\item[(2)]  $p(T(x),T(y))< p(x,y)\; \forall \;x\in A_i, y\in
A_{i+1}$ for $1\leq i\leq k$.
\end{itemize}
Then $T$ has a unique fixed point.
\end{theorem}

\begin{proof}
Let $A_1$ be $0$--compact and $\delta=p(A_1,A_k)=\inf\{p(x,y):x
\in A_1, y \in A_k\}$. From the definition of $\delta$ there exist
sequences $\{x_n\}\subset A_1$ and $\{u_n\}\subset A_k$ such that
$$p(x_n,u_n)\leq \delta+ \frac{1}{n}.$$
By the $0$--compactness of $A_1$ we may assume that there exists
$x_0 \in A_1$ such that the limit $\lim_{n\rightarrow \infty}
p(x_n,x_0)=p(x_0,x_0)=0$. Then by the triangle inequality it
follows that $\lim_{n\rightarrow \infty}p(x_0,u_n)=\delta$. Let
$\delta >0$. Then
\begin{equation}\label{ed1}
    p(T^{k+1}(x_0), T^{k+1}(u_n))<...< p(x_0,u_n).
\end{equation}
Since the sequence $\{T^{k+1}(u_n)\}$ is in $A_1$ and $A_1$ is
$0$--compact, we may assume that there exists $z\in A_1$ such that
 $\lim_{n\rightarrow \infty}p(T^{k+1}(u_n),z)=p(z,z)=0$.
 By (\ref{ed1}) and Lemma \ref{continuity of pm} we conclude that
$$p(z,T^{k+1}(x_0))\leq \delta.$$
It follows that
$$p(T^{k-1}(z),T^{2k}(x_0))< \delta$$
but since $T^{k-1}(z) \in A_k$ and $F^{2k}(x_0) \in A_1$ we obtain
a contradiction. Therefore, we conclude that $\delta=0$ and $A_1
\cap A_k \neq \emptyset$. Thus, by assumption (1) of the theorem,
$A_1 \cap A_2\neq\emptyset$.

 We now consider the sets $B_1=A_1\cap A_2$, $B_2=A_2\cap A_3,\dots,B_k=A_k\cap A_1$. In view of
 assumption (1) these sets are all nonempty (and closed) and $B_1$ is $0$-compact.
 Thus the assumptions (1) and (2) of the theorem hold for $T$ and the family $\{B_i\}_{i=1}^k$. By repeating
 the arguments just given we arrive at
$$B_1\cap B_k\neq\emptyset.$$
it follows that $A_1 \cap A_2\cap A_3\neq \emptyset$. Continuing
step--by--step, we conclude that $A :=\cap_{i=1}^k\neq \emptyset$.
The uniqueness, however,  follows from the fact that any fixed
point of $T$ necessarily lies in $A :=\cap_{i=1}^k$ which is
clearly obtained by assumption (1).
\end{proof}

\end{document}